\documentclass[12pt]{amsproc}
\usepackage{amssymb}

\usepackage{amsfonts,amsmath,amsthm}
\usepackage{amssymb}
\usepackage[utf8]{inputenc}
\usepackage[english]{babel}
\usepackage{graphicx}
\usepackage{array}
\usepackage{url}
\usepackage{hyperref}
\usepackage{booktabs}
\usepackage[all]{xy}

\newtheorem{theorem}{Theorem}[section]
\newtheorem*{theorem*}{Theorem}
\newtheorem{lemma}[theorem]{Lemma}
\newtheorem{proposition}[theorem]{Proposition}
\newtheorem{corollary}[theorem]{Corollary}
\newtheorem{definition}[theorem]{Definition}
\theoremstyle{remark}
\newtheorem{remark}[theorem]{Remark}
\newtheorem{note}[theorem]{}
\newtheorem{example}[theorem]{Example}

\def\coc{\!\!\sim}

\def\Q{\mathbb Q}
\def\K{\mathbb K}
\def\F{\mathbb F}
\def\Z{\mathbb Z}
\def\N{\mathbb N}
\def\norma{\|\cdot\|}

\def\XnX{(X,\|\cdot\|_X)}

\def\YnY{(Y,\|\cdot\|_Y)}

\def\Iso{\operatorname{Iso}}

\title{Isometries of ultrametric normed spaces}

\begin{document}

\author
{Javier~Cabello~Sánchez, José~Navarro~Garmendia*}
\address{*Corresponding author: Departamento de Matem\'{a}ticas, Universidad de Extremadura, 
Avenida de Elvas s/n, 06006; Badajoz. Spain}
\email{coco@unex.es, navarrogarmendia@unex.es}


\begin{abstract} 
We show that the group of isometries of an ultrametric normed space can be seen 
as a kind of a fractal.  Then, we apply this description to study ultrametric 
counterparts of some classical problems in Archimedean analysis, such as the so 
called {\it Problème des rotations de Mazur} or Tingley's problem.

In particular, it turns out that, in contrast with the case of real normed spaces, 
isometries between ultrametric normed spaces can be very far from being linear.
\end{abstract}

\maketitle


\section{Introduction}

The study of isometries dates back to the classification of them as translation, 
reflections and rotations in the two-dimensional Euclidean space and comprises a 
huge variety of results, that include one of the most ancient and celebrated 
results in the theory of Banach spaces: the Mazur--Ulam Theorem states that 
every onto isometry between real Banach spaces is affine. 

In \cite{NA-Moslehian}, there was an attempt to prove an ultrametric version of the 
Mazur--Ulam Theorem introducing the notion of non-Archimedean strictly convex space. 
Nevertheless, Professor A.~Kubzdela observed  some  years later that non-Archimedean 
strictly convex spaces are a {\em rarity} (see~\cite{Kubzdela}). 

Quite recently,  the authors of the present work have shown that any 
attempt to obtain ultrametric versions of the Mazur--Ulam Theorem via strictly 
convex spaces is doomed to fail  (see~\cite{JCSJNGNLA}).

The aim of this paper is to go a little further and analyse to which extent the 
behaviour of isometries in ultrametric analysis is far from their behaviour 
in real analysis.  Our main result is Theorem \ref{Fractal},  whose proof is elementary, and that exhibits a fractal structure of the group of isometries of an ultrametric normed space.

As a consequence, we prove in Corollary \ref{CorolarioMU} that the only ultrametric normed spaces that possess a Mazur--Ulam like property are the trivial examples whose bijections are always affine; namely, the one-dimensional normed spaces over 
$\Z/2\Z$ or $\Z/3\Z$, and the two-dimensional normed spaces over $\Z/2\Z$.

\subsection{Preliminaries}

The first definitions are a commonplace in ultrametric analysis, and we reflect 
them just for the sake of completeness, see \cite{NA-Moslehian}. 

\begin{definition}\label{defNAF}
A {\em non-Archimedean} (or {\em ultrametric}) {\em field} is a field $\K$ equipped with a function $|\cdot|:\K\to[0,\infty)$ such that 
\begin{enumerate}
\item[i)] $|\lambda | = 0$ if and only if $\lambda = 0$, 
\item[ii)] $|\lambda \mu| = |\lambda| |\mu|$, 
\item[iii)] $|\lambda + \mu | \leq max\{|\lambda |, |\mu|\}$ for all $ \lambda , \mu \in \K$. 
\end{enumerate}
\end{definition}

\begin{definition}\label{defNAN}
An {\em ultrametric normed space} is a linear space $X$ over a non-Archimedean 
field $(\K , |\cdot|)$ that is endowed with an ultrametric norm; that is to 
say, endowed with a function $\norma:X\to[0,\infty)$ such that 
\begin{enumerate}
\item[iv)] $\|x\| = 0$ if and only if $x = 0$, 
\item[v)] $\|\lambda  x\| = |\lambda |\|x\|$,  for all $\lambda \in \K, x\in X$.     
\item[vi)] $\|x + y\| \leq max\{\|x\|,\|y\|\}$ for all $x,y\in X$.     
\end{enumerate}
\end{definition}

Definition~\ref{defNAN} endows $X$ with a structure of metric 
space, thus allowing us to study the isometries between (subsets of) ultrametric normed spaces: 

\begin{definition} Let $A \subseteq X$ and $B \subseteq Y$ be subsets of two ultrametric spaces $X, Y$. A map $\,f \colon A \to B\,$ 
is an {\em isometry} if it is bijective and, for any $x,y \in A$,
$$\|f(y)-f(x)\|=\|y-x\|.$$ 
\end{definition}

The inversion through the origin, $x \mapsto -x$, is always an isometry $X \to X$, as so they are translations by a vector $\, x \mapsto x + z\,$.

\medskip
Finally, let us also recall the following well-known result:

\begin{lemma}[\cite{schikhoff}, 8.C]\label{Isosceles} 
On an ultrametric normed space $X$, every triangle 
is {\em isosceles in the big}, i.e., for every triplet of points 
$x, y, z\in X$, $\|x-z\|<\|y-z\|$ implies $\|y-x\|=\|y-z\|$. 
\end{lemma}
%
%

\section{Fractality of isometries}

Let $X$ be an ultrametric normed space. Let us denote the open ball and the sphere of radius $r \in (0 , \infty)$ as
$$B(x,r):=\{y\in X\colon\|y-x\|<r\}\quad,\quad S_X(r):=\{x\in X:\|x\|= r\}\ , $$ 
and let us also convene that 
$$B(x, \infty) := X \, . $$

\begin{proposition}\label{SecondStatement} Let $r \in (0,\infty]$, and consider an arbitrary family:
$$ f_{r'} \colon S_X (r') \to S_X(r') \, $$
of isometries on each sphere $S_X(r')$, for $r' \in (0, r)$.

Then, the map:
$$ f \colon B(0,r) \longrightarrow B(0,r) \quad , \quad  f(x) := \left\{
\begin{array}{c l}
f_{r'}(x) , & \text{ if } x \in S_X(r') \\
0, & \text{ if } x = 0 
\end{array}
\right. $$  is a centred isometry (i.e., an isometry such that $f(0) = 0$).
 \end{proposition}

\begin{proof} Observe that $\| x \| = \| f(x) \|$, for any $x \in B(0,r)$. 

Thus, for any pair of points $x,y\in B(0,r)$,
\begin{itemize}
\item if $\| x \| = \|y \| = r'$, then $$\|y - x \| = \| f_{r'}(y)-f_{r'}(x)\| = \| f(y) - f(x) \|  \ . $$

\item if $\|x\|<\|y\|$, then $\|f(x)\|<\|f(y)\|$ and 
$$ \|f(y)-f(x)\|=\|f(y)\|=\|y\|=\|y-x\|.$$

\item if $\|x\|>\|y\|$, then $\|f(x)\|>\|f(y)\|$ and 
$$ \|f(y)-f(x)\|=\|f(x)\|=\|x\|=\|y - x\|.$$
\end{itemize} \end{proof}

\medskip
The following conditions define equivalence relations on $X$, for any $r \in (0, \infty ]$,
$$ x \sim_r y \ : \Leftrightarrow \ \|y - x\| < r \ . $$

The equivalence classes for these relations are precisely the open balls $ B(x,r)$. Therefore,  an ultrametric normed space $\,X$, as well as any other subset $\,A \subset X$, canonically decomposes as a disjoint union of open balls of any fixed radius $r$.

In particular, if we write $S/\coc_r\,$ to abbreviate the quotient space $\,S_X(r)/\coc_r$,  and we denote 
$$ B_{\bar{x}} (r) := B(x,r) \ , \ \mbox{where $x \in X$ is any representative of $\bar{x} \in S/\coc_r $} \, , $$
any sphere decomposes as a disjoint union of open balls:
\begin{equation}\label{DescomposicionS}
S_X(r) \, = \, \bigsqcup_{\bar{x} \in S/\sim_r} B_{\bar{x}}( r) \ .
\end{equation}

\begin{proposition}\label{ThirdStatement} Let $r \in (0,\infty)$ be a positive radius.

If $\,\sigma$ is a permutation of the quotient set $S/\!\coc_r$ and 
$\,\{ \varphi_{\bar{x}} \colon B_{\bar{x}}(r) \to B_{\sigma( \bar{x})} (r)\}_{\bar{x} \in S/\sim_r}\,$ 
is an arbitrary family of isometries, then the disjoint union of these isometries
$$ S_X(r) \ \xrightarrow{\ \ f \ \ } \  S_X(r)
\quad , \quad  f(x) := \varphi_{\bar{x}} (x) \ ,  $$  
is an isometry of the sphere $S_X(r)$.
 \end{proposition}

\begin{proof} Any two points of the sphere $\, x , y \in S_X(r)\,$ satisfy $\| y - x \| \leq r$ (because of Lemma \ref{Isosceles}), so that the distance between points $x,y \in S_X(r)$ in different open balls of (\ref{DescomposicionS}) is exactly $r$.

Therefore, for any pair of points $x, y \in S_X (r)$, 
\begin{itemize}
\item if $\| y - x \| < r$, then $\bar{x} = \bar{y}$ and
$$\| f(y) - f(x) \|  = \|  \varphi_{\bar{x}} (y) - \varphi_{\bar{x}} (x) \| = \| y - x \| \ . $$

\item if $\| y - x\| = r$, then $\bar{x} \neq \bar{y}$.  Hence, $\sigma (\bar{x}) \neq \sigma ( \bar{y})$, so that $B_{\sigma (\bar{x})} (r) \neq B_{\sigma (\bar{y})}(r)$ and, in particular, $f(x)$ and $f(y)$ lie in different open balls of (\ref{DescomposicionS}); that is to say, $\| f(y) - f(x) \|  =  r$.
\end{itemize}
\end{proof}









We can summarize these results as follows:

\medskip
\begin{theorem}\label{Fractal}
The following statements hold:
\begin{enumerate}
\item (Isometries between balls) Let $B_1$ and $B_2$ be open balls with the same (possibly infinite) radius $r\in (0,\infty]$. A map $f\colon B_1 \to B_2$ is an isometry if and only if there exist translations $\tau , \tau'$ and a centred isometry $f_c\colon B(0,r) \to B(0,r)$ such that
$$ f \, = \, \tau' \circ f_c \circ \tau \ . $$

\item (Centred isometries between balls) A centred map $\,f \colon B(0,r) \to B(0,r) \,$, with $r \in (0,\infty]$, is an isometry if and only if 
\begin{enumerate}
\item It preserves norms: $\| f(x) \| = \| x \|$, for any $x\in B(0,r)$.

\item It is an isometry on each sphere: 
$$ f \colon S_X (r') \to S_X(r') \ \mbox{ is an isometry, for any } r' \in (0, r). $$
\end{enumerate}

\item (Isometries between spheres) A map $\,f\colon S_X(r) \to S_X(r)\,$, with $r \in (0, \infty)$, is an isometry if and only if there exist a permutation $\,\sigma$ of the quotient set $S/\coc_r$ and isometries $\,\varphi_{\bar{x}} \colon B_{\bar{x}}( r) \to B_{\sigma(\bar{x})}(r)\,$, for $\bar{x} \in S/\coc_r$, such that $f$ is the disjoint union of the $\varphi_{\bar{x}}$:
$$ S_X(r) = \bigsqcup_{\bar{x} \in S/\sim_r} B_{\bar{x}}(r) \ \xrightarrow{\ \ f = \sqcup\, \varphi_{\bar{x}}\ \ } \ \bigsqcup_{\bar{x} \in S / \sim_r} B_{\bar{x}}(r ) = S_X(r) \ . $$

\end{enumerate}
\end{theorem}

\begin{proof} For the first statement, assume $f \colon B_1 \to B_2$ is an isometry between open balls of the same radius $r \in (0, \infty]$. 

For {\it any} pair of points $x \in B_1$ and $y \in B_2$, the translation by the vector $y-x$ is an isometry $ \tau_{y-x} \colon B_1 \to B_2$: in fact,  for any $z \in B_1$, $\| \tau_{y-x} (z) - y \| = \| z - x \| < r$, so that $\tau_{y-x} (z) \in B_2$.

Therefore, if we take any point $x \in B_1$,  the translations $\tau_{-x} \colon B_1 \to B(0,r)$ and $\tau_{f(x)} \colon B (0, r) \to B_2$ are well-defined, and $f_c$ is the unique isometry that makes the square commutative:
$$
\xymatrix{B_1 \ar[rr]^-{f}_{\sim} \ar[d]_{\tau_{-x}} & & B_2 \\ 
B(0,r) \ar@{.>}[rr]^-{f_c} & & B(0,r) \ar[u]_{\tau_{f(x)}}. }
$$

\medskip
Finally, both the second and third statements are reformulations of Proposition \ref{SecondStatement} and Proposition \ref{ThirdStatement}, respectively.
\end{proof}

\begin{example}\label{descripcion}
In order to give an explanation of the name of this Section, let us describe 
the group of centred autoisometries of a quite simple space as $X=\Q^2_3$ when 
we consider the following valuation and norm: 
$$\left|3^n\frac ab\right|=3^{-n};\quad \|(\lambda,\mu)\|=\max\{|\lambda|,|\mu|\}, $$
for every $n, a, b\in\Z$ such that $a$ and $b$ are coprime and $\lambda,\mu\in\Q_3$. 

The norm $\norma$ only takes values in $\{0\}\cup\{3^n:n\in\Z\}$, so the second 
part of Theorem~\ref{Fractal} implies that we only need to study the autoisometries 
$\tau_n:S_X(3^n)\to S_X(3^n)$ for each $n\in\Z$. Later, we can glue these isometries to 
obtain $\tau:X\to X$ defined as $\tau(x)=\tau_n(x)$ when $\|x\|=3^n$, $\tau(0)=0$. 

It is clear that any isometry $\tau_n:S_X(3^n)\to S_X(3^n)$ is associated with 
$\tau_0:S_X\to S_X$ via the dilations $S_X(3^n)\to S_X$ and 
$S_X\to S_X(3^n)$, so the structure of $\Iso_X$ is determined by the 
estructure of $\Iso_{S_X}$: we only need to find one isometry $S_X\to S_X$ for 
each integer $n$ and then glue them all in their corresponding spheres $S_X(3^n)$. 
But we can also take some bijection $\Z\to\N\cup\{0\}$, like the one defined as 
$n\mapsto 2n,\ \ -n\mapsto (-2n-1)$ for every $n\in\N$ and $0\mapsto 0$. This 
way, we have a correspondence between $\Iso_X$ and $\Iso_{B_X}$ thanks to 
$\Iso_{S_X}$. So, let us analyse what happens with $\Iso_{S_X}$. We can 
decompose $S_X$ as in (\ref{DescomposicionS}), but we have the nice feature 
that $B_X(x,1)=B_X[x,1/3]$ for every $x\in X$ --because $\|y\|<1$ is equivalent 
to $\|y\|\leq 1/3$-- and in particular we get 
\begin{equation}\label{ArtDeco}
S_X=\bigsqcup_{\bar{x} \in S/\sim_1} B_X[x,1/3].
\end{equation}
The decomposition given in (\ref{ArtDeco}) is much more simple than it seems. 
Namely, as the space that we are dealing with is $X=\Q^2_3$, the quotient 
$S/\coc_1$ contains only a finite number $k$ of equivalence classes. 
So, we only need to choose some permutation $\sigma\in S_k$ ($S_k$ 
stands for the symmetric group) and $k$ isometries 
$\varphi_i:B_X[0,1/3]\to B_X[0,1/3]$. Of course, the structure of 
$\Iso_{B_X[0,1/3]}$ and the one of $\Iso_{B_X}$ are the same, so we can restrict ourselves 
to the study of $\Iso_{B_X}$\ldots
\end{example}


\section{Some consequences}

\subsection{Isotropy of ultrametric spaces} 

One of the most famous open problems in functional analysis is the following: 
\begin{note}[Problème des rotations de Mazur]
Let $X$ be a separable real or complex Banach space and suppose that, for every 
$x,y\in S_X$ there is an onto isometry $f:X\to X$ such that $f(x)=y$. 
Does this imply that $X$ is an inner product space? 

This was solved in the positive for finite-dimensional spaces and in the negative 
for non-separable spaces, but it is still open in the above form. 
\end{note}

In the ultrametric setting, though, the group of isometries always acts 
transitively on the spheres: 

\begin{corollary} Let $X$ be an ultrametric normed space. For any pair of 
vectors $x,y \in X$ with the same norm, there 
exists a centred isometry $f \colon X \to X$ such that $f(x) = y$.
\end{corollary}

\begin{proof}
For any $x,y \in S_X(r)$, there exists an isometry $\varphi \colon S_X(r) 
\to S_X(r)$ such that $\varphi(x) = y$ (Proposition \ref{ThirdStatement}).

A global isometry is a gluing of isometries on each sphere 
(Proposition~\ref{SecondStatement}) so we are done.
\end{proof}

\begin{remark}
We distinguish {\em non-Archimedean} from {\em ultrametric}, although both terminologies are sometimes 
seen as equivalent. But they do not need to be. Actually, the definition of normed space over a valuated field 
that the reader can find in Bourbaki's monumental work~\cite{Bourbaki} does not 
involve the ultrametric inequality. Both concepts {\em are} equivalent when 
one refers to valuated fields (because a multiplicative norm $|\cdot|:\K\to[0,\infty)$ satisfies the ultrametric inequality  if and only if the image of the natural map $\mathbb{N} \to \K $ is bounded, [\cite{schikhoff}, 8.2]),  but we have found no reason whatsoever 
to limit the non-Archimedean norms to ultrametric norms. 

In fact, recent work (\cite{JCSFCF1}) points out that it may be  interesting to study the structure that usual (not ultrametric) norms give to linear spaces 
over non-Archimedean fields. 
\end{remark}

\subsection{Tingley's Problem} 

Due to the Mazur--Ulam Theorem and Mankiewicz's Theorem, that states that 
every onto isometry $\tau:B_X\to B_Y$ between the unit balls of two real Banach 
spaces extends to a linear isometry $\widetilde\tau:X\to Y$, it stands to reason 
to wonder whether every onto isometry $\tau:S_X\to S_Y$ between the unit spheres 
of real Banach spaces extends to a linear isometry, too. This is known as 
Tingley's Problem, and is receiving a lot of attention nowadays, see, e.g., 
\cite{JCSRefJMAA, PeraltaSurvey} and the question can be generalised to 
{\em does every onto isometry between the boundaries of open convex subsets 
extend to a linear isometry between the spaces?}, see\cite{JCSJordan}; 
and to finding minimal requirements on the subset where the isometry is 
defined, see~\cite{CuetoPeraltaJB}. 
One key in this question is that $\tau$ extends linearly if and only if 
it extends to an isometry, so we have that the question is equivalent, 
in real Banach spaces, to any of the following:
\begin{itemize}
\item[1] Does every onto isometry $\tau:S_X\to S_Y$ extend to an isometry
$\widetilde\tau:X\to Y$?
\item[2] Does every onto isometry $\tau:S_X\to S_Y$ extend to a linear map 
$\widetilde\tau:X\to Y$?
\end{itemize}
In ultrametric normed spaces, it is clear that the answer to the first question 
is {\em yes, with conditions} and the answer to the second one is {\em not even close}. 
Please be aware that, for the question to make sense we need to choose 
a nonempty sphere instead of the unit sphere. 

\begin{proposition}\label{Tingley}
Let $\XnX$ and $\YnY$ be ultrametric normed spaces over some valuated field 
$\K$, $r,r'>0$ such that $rS_X\neq\emptyset$ and $\tau:rS_X\to r'S_Y$ an 
isometry. If the valuation is not trivial, then $r'=r$ and $\tau$ extends to 
an isometry $\widetilde\tau:X\to Y$. 
\end{proposition}

\begin{proof}
Let $x_0\in rS_X$. The ultrametric inequality implies that $B_X(x_0,r)\subseteq S_X(0,r)$. 
It is clear that $\tau:B_X(x_0,r)\to B_Y(\tau(x_0),r)$ is also an onto isometry, so 
we may compose $\tau$ with translations to get a (centred) isometry 
$\widetilde\tau:B_X(0,r)\to B_Y(0,r)$ defined as $\widetilde\tau(x)=\tau(x+x_0)-\tau(x_0)$. 
Gluing $\tau$ with $\widetilde\tau$ we obtain an isometry, which we will not rename, 
$$\widetilde\tau:B_X[0,r]\to B_Y[0,r].$$

Take some $\alpha\in\K$ such that $a=|\alpha|>1$ -- it exists because the 
valuation is not trivial. For any $x\in B_X[0,ra]\setminus B_X[0,r]$ 
(equivalently, $r<\|x\|_X\leq ra$), define 
$\widetilde\tau(x)=\alpha\tau(x/\alpha)$. It is clear that 
$\widetilde\tau:S_X(0,t)\to S_Y(0,t)$ is a well-defined onto isometry for 
every $t\in (r,ra]$. This implies that $\widetilde\tau:B_X[0,ra]\to B_Y[0,ra]$ 
is a centred onto isometry. The same way we can define 
$\widetilde\tau(x)=\alpha^n\tau(x/\alpha^n)$ whenever 
$x\in B_X[0,ra^n]\setminus B_X[0,ra^{n-1}]$, thus finishing the proof. 
\end{proof}

\begin{remark}
If the valuation of $\K$ is trivial, then we can find at least two pathological 
behaviours. Namely, if $(\Z/2\Z)^2$ is endowed with the ultrametric norms 
$$\|(a,b)\|_X=\max\{|a|,2|b|\}, \|(a,b)\|_Y=\max\{|a|,3|b|\}$$ 
then both unit spheres consist in the singleton $\{(1,0)\}$ but the isometry 
$\tau:S_X\to S_Y$ defined as $\tau(1,0)=(1,0)$ does not extend. The same happens 
if we define the same norm over any other two-dimensional space over a field 
whose valuation is trivial. 

Moreover, if we endow $(\Z/2\Z)^2$ with 
$$\|(a,b)\|_X=\max\{|a|,2|b|\}, \|(a,b)\|_Y=\max\{2|a|,3|b|\}$$ 
then both $S_X$ and $2S_Y$ contain exactly one point. This means that they are 
trivially isometric, but, with the notations of Proposition~\ref{Tingley}, we 
have $r'\neq r$. 
\end{remark}

\subsection{Non-existence of Mazur-Ulam type theorems}

\begin{example}
The map $x \to -x$ is always a centred isometry $X \to X$ and, for any fixed $x_0$, the 
map $x\mapsto x+x_0$ is an isometry of any sphere $S_X(r)$ with $r>\|x_0\|$.

Therefore, for any $r>\|x_0\|$ such that $S_X(r)\neq\emptyset$, 
the maps $f, \widetilde{f}\colon X\to X$ defined as
$$ f(x) = \left\{
\begin{array}{c l}
-x, & \text{ if } x \in S_X(r)\\
x, & \text{ if } x \notin S_X(r)
\end{array}
\right.
\quad , \quad 
\widetilde{f}(x) = \left\{
\begin{array}{c l}
x+x_0, & \text{ if } x \in S_X(r)\\
x, & \text{ if } x \notin S_X(r)
\end{array}
\right.$$ are centred isometries.

Observe that, whenever they are distinct from the identity, $f$ and $\widetilde{f}$ are non-linear maps.
\end{example}

\begin{example}[Isometries of one-dimensional spaces over finite fields]
Any finite field $\mathbb{F}_q$ is a non-Archimedean field in a unique way: 
equipped with the trivial valuation $|\lambda | = 1$, for any non-zero element 
$\lambda \in \mathbb{F}_q$. 
As a result, any bijective map $f \colon \mathbb{F}_q \to \mathbb{F}_q$ is an isometry. 

In the same way, every ultrametric norm on a one-dimensional linear space $X_1$ over $\mathbb{F}_q$ is trivial; i.e., for any such norm there exists 
$a \in (0, \infty)$ such that $\|x\|= a $ for any non-zero $x \in X_1$. 
Consequently, any bijective map $f \colon X_1 \to X_1$ is an isometry.
\end{example}

\medskip
\begin{corollary}\label{CorolarioMU} Let $X$ be an ultrametric normed space. If every centred isometry $\,f\colon X\to X\,$ is a linear  map, then one of the following options holds:
\begin{enumerate}
\item $X=0$.
\item $X$ is a one-dimensional linear space over $\Z/2\Z$. 
\item $X$ is a one-dimensional linear space over $\Z/3\Z$. 
\item $X$ is a two-dimensional linear space over $\Z/2\Z$. 
\end{enumerate}
\end{corollary}

\begin{proof}
If the norm on $X$ is trivial, 
then any bijection 
$f \colon X \to X$ is an isometry. If $X$ has more than four elements, 
then we may take different $x_1,x_2\in X\setminus\{0\}$ such that 
$x_3=x_1+x_2\neq 0$ and $x_4\not\in\{x_1,x_2,x_3,0\}$ and define 
$f:X\to X$ as $f(x)=x$ for every $x\not\in\{x_3,x_4\}, f(x_3)=x_4, f(x_4)=x_3$. 
This is obviously a nonlinear onto isometry from $X$ to $X.$

\medskip
If $\|\cdot\|$ is not trivial, then we have two options: apart from the value 0, 
either $\|\cdot\|$ takes exactly two values or it takes at least three. 

If there are $x_0,x_1,x_2$ such that $0<\|x_0\|<\|x_1\|<\|x_2\|$, 
then we may define the mapping $f:X\to X$ as
$$f(x)=
\left\{
\begin{array}{c l}
x, & \text{ if }\|x\|\neq\|x_1\|\\
x+x_0, & \text{ if }\|x\|=\|x_1\|
\end{array}
\right..$$
Proposition~\ref{SecondStatement} implies that $f$ is an onto isometry and it is not 
linear because $f(x_2+x_1)=x_2+x_1$, $f(x_2)=x_2,$ $f(x_1)=x_1+x_0$. 

\medskip 
The only option left is that $\|\cdot\|$ takes exactly two positive values, say 
there are $x_0,x_1\in X$ such that $0<\|x_0\|<\|x_1\|$. 

Consider the map
$$ f \colon X \to X, \quad  f(x) := \begin{cases}
\ \ x \ ,  \qquad \mbox{ if } \| x \| =\|x_1\|\\
- x \ ,  \qquad \mbox{ if } \| x \| =\|x_0\|
\end{cases} .  $$ 

\medskip
This map is a bijection such that $f(0) = 0$ and it is clear that fulfils the 
conditions in Proposition~\ref{SecondStatement}, so it is an isometry. 

\medskip

Moreover, this map is not linear: 
$$f(x_1+x_0)-f(x_1)-f(x_0)=x_1+x_0-x_1+x_0=2x_0\neq 0\ldots $$
unless $2=0$ in $\K$. So, the only problem we have right now is 
that $\K$ has characteristic 2, its valuation is trivial (if it is not, then 
the valuation takes infinitely many values, and so does the norm), and the norm 
takes exactly two positive values. 

If $X$ fulfils everything, then take $0<\|x_1\|<\|x_2\|$. 
As every triangle is {\em isosceles in the big}, for every pair $x,y$ such that 
$\|x\|=\|y\|=\|x_1\|$ we have $\|x+y\|\leq\|x_1\|$. Furthermore, the valuation 
of $\K$ is trivial, so 
$$X_1=\{x\in X:\|x\|\leq\|x_1\|\}=\{x\in X:\|x\|=\|x_1\|\}\cup\{0\}$$
is closed under addition and scalar multiplication and this means that it is 
a vector subspace of $X$ with trivial norm, so if it has only affine 
isometries it must have 4 elements or less. If there is some nonlinear isometry 
$X_1\to X_1$ that sends 0 to 0 then we just need to extend it to $X$ by identity, 
so our only problem is that $X_1$ has at most four elements. In this case,  
$\K=\Z/2\Z$ implies that $X$ is (linearly) isometric to either $\Z/2\Z$ or $(\Z/2\Z)^2$ 
and $\K=\F_4$ implies that $X$ is (linearly) isometric to $\K$. 

In the first case, it is easy to see that a bijection $f:X\to X$ is an isometry 
if and only if $f(x+x_1)=f(x)+x_1$ for every $x\in X$. As the dimension of $X$ 
is at least 3, there are $x,y,z\not\in X_1$ such that $z=x+y$. Now, the map 
$f:X\to X$ defined as $f(x)=x$ whenever $x\not\in\{z,z+x_1\}$, 
$f(z)=z+x_1, f(z+x_1)=z$ is an isometry and it is not affine. 

In the second case, we only need to consider some linear isometry 
$\tau:X_1\to X_1$ that does not send $x_1$ to $x_1$ --it exists 
because $X_1$ contains four elements-- and define $f:X\to X$ as
$$
f(x)=\left\{
\begin{array}{c l}
x & \text{ if }x\not\in X_1\\
\tau(x) & \text{ if }x\in X_1
\end{array}
\right..
$$
This is an isometry by Proposition~\ref{SecondStatement}. Furthermore it is not linear because 
$f(x_2+x_1)=x_2+x_1,$ $f(x_1)=\tau(x_1)\neq x_1$ and $f(x_2)=x_2$. 

To finish the proof, we must observe that if we consider the four element field
$\F_4$, that can be seen as $\F_4=\big((\Z/2\Z)[x]\big)/(x^2+x+1)$, then every 
centred bijection $\F_4\to\F_4$ is an additive isometry. But it is clear that
the map $f:\F_4\to\F_4$ defined as
$$f([0])=[0], f([1])=[1],f([x])=[x^2],f([x^2])=[x]$$
is not $(\F_4)$-linear.
\end{proof}

\section*{Acknowledgments}
Supported in part by DGICYT project PID2019-103961GB-C21 (Spain), 
Feder funds and Junta de Extremadura program IB-18087.


\bibliographystyle{abbrv}

\bibliography{Fractal}

\end{document}